\documentclass[a4paper]{amsart}
\usepackage{amssymb, enumitem}
\usepackage[all]{xy}
\usepackage{hyperref, aliascnt, graphicx}

\newtheorem{lma}{Lemma}[section]

\newaliascnt{thmCt}{lma}
\newtheorem{thm}[thmCt]{Theorem}
\aliascntresetthe{thmCt}

\newaliascnt{corCt}{lma}
\newtheorem{cor}[corCt]{Corollary}
\aliascntresetthe{corCt}

\newaliascnt{prpCt}{lma}
\newtheorem{prp}[prpCt]{Proposition}
\aliascntresetthe{prpCt}

\newtheorem*{thm*}{Theorem}
\newtheorem*{cor*}{Corollary}
\newtheorem*{prop*}{Proposition}

\theoremstyle{definition}

\newaliascnt{pgrCt}{lma}
\newtheorem{pgr}[pgrCt]{}
\aliascntresetthe{pgrCt}

\newaliascnt{dfnCt}{lma}

\aliascntresetthe{dfnCt}

\newaliascnt{rmkCt}{lma}
\newtheorem{rmk}[rmkCt]{Remark}
\aliascntresetthe{rmkCt}

\newaliascnt{rmksCt}{lma}

\aliascntresetthe{rmksCt}

\newaliascnt{exaCt}{lma}
\newtheorem{exa}[exaCt]{Example}
\aliascntresetthe{exaCt}

\newaliascnt{qstCt}{lma}
\newtheorem{qst}[qstCt]{Question}
\aliascntresetthe{qstCt}

\newaliascnt{pbmCt}{lma}
\newtheorem{pbm}[pbmCt]{Problem}
\aliascntresetthe{pbmCt}

\newtheorem{thmx}{Theorem}

\def\today{\number\day\space\ifcase\month\or   January\or February\or
   March\or April\or May\or June\or   July\or August\or September\or
   October\or November\or December\fi\   \number\year}

\newcommand{\CC}{{\mathbb{C}}}
\newcommand{\RR}{{\mathbb{R}}}
\newcommand{\Bdd}{{\mathcal{L}}}

\newcommand{\Isom}{{\mathrm{Isom}}}
\newcommand{\id}{{\mathrm{id}}}
\newcommand{\ev}{{\mathrm{ev}}}
\newcommand{\re}{{\mathrm{Re}}}

\newcommand{\ca}{$C^*$-algebra}

\DeclareMathOperator{\Rep}{Rep}
\newcommand{\ltArensProd}{\square}
\newcommand{\rtArensProd}{\lozenge}

\newcommand{\Pelc}{Pe{\l}czy\'{n}ski}
\newcommand{\weakStar}{weak${}^*$}

\title{Extending representations of Banach algebras to their biduals}
\date{\today}

\author[Eusebio Gardella]{Eusebio Gardella}
\address{Eusebio Gardella. Mathematisches Institut, Universit\"at M\"unster, Einsteinstr.~62, 48149 M\"unster, Germany.}
\email{gardella@uni-muenster.de}
\urladdr{http://pages.uoregon.edu/gardella/}
\author{Hannes Thiel}
\address{Hannes Thiel. Mathematisches Institut, Universit\"at M\"unster, Einsteinstr.~62, 48149 M\"unster, Germany.}
\email{hannes.thiel@uni-muenster.de}
\urladdr{www.math.uni-muenster.de/u/hannes.thiel/}
\thanks{Part of this research was conducted while the authors were taking part in the Research Program \emph{Classification of operator algebras, complexity, rigidity and dynamics}, held at the Institut Mittag-Leffler, between January and April of 2016.
The first named author was partially supported by a Postdoctoral Research Fellowship from the Humboldt Foundation.
The authors were partially supported by the Deutsche Forschungsgemeinschaft (SFB 878 Groups, Geometry \& Actions).}

\subjclass[2010]{
Primary:
46H15,   
Secondary:
46E30,  
46L05,  
47L10.  
}
\keywords{}

\begin{document}

\begin{abstract}
We show that a representation of a Banach algebra $A$ on a Banach space $X$ can be extended to a canonical representation of $A^{**}$ on $X$ if and only if certain orbit maps $A\to X$ are weakly compact.
When this is the case, we show that the essential space of the representation is complemented if $A$ has a bounded left approximate identity.
This provides a tool to disregard the difference between degenerate and nondegenerate representations.

Our results have interesting consequences both in $C^*$-algebras and in abstract harmonic analysis.
For example, a \ca{} $A$ has an isometric representation on an $L^p$-space, for $p\in[1,\infty)\setminus\{2\}$, if and only if $A$ is commutative.
Moreover, the $L^p$-operator algebra of a locally compact group is universal with respect to arbitrary representations on $L^p$-spaces.
\end{abstract}

\maketitle

\section{Introduction}

Let $A$ be a Banach algebra.
By the groundbreaking work of Arens, multiplication on $A$ can be extended in two natural ways to a multiplication on its bidual $A^{**}$, called the left Arens product $\ltArensProd$ and the right Arens product $\rtArensProd$.
Each of these multiplications give $A^{**}$ the structure of a Banach algebra, and the map $\kappa_A\colon A\to A^{\ast\ast}$ is multiplicative.
Despite being complicated objects, biduals of Banach algebras are usually accessible through powerful tools that are not
at one's disposal to study the Banach algebras themselves: among others, arguments involving polar decompositions can be performed,
and a much more general type of functional calculus is available.
The technical advantages enjoyed by biduals are regularly exploited in the theory of $C^*$-algebras, where Kaplansky's density theorem often allows one to transfer information back to $A$.

It therefore becomes relevant to know in what situations one can work in the much more flexible context of the bidual algebra.
More concretely, we are interested in knowing when a representation of $A$ extends to $A^{\ast\ast}$.

\begin{pbm}
\label{pbm:extendRepr}
Given a representation $\varphi\colon A\to\Bdd(X)$, find conditions that guarantee that there exists a representation
$\tilde{\varphi}\colon (A^{**},\rtArensProd)\to\Bdd(X)$, or $\tilde{\varphi}\colon (A^{**},\ltArensProd)\to\Bdd(X)$, making the following diagram commute:
\[
\xymatrix{
A^{**} \ar[dr]^{\tilde{\varphi}} \\
A \ar@{^{(}->}[u]^{\kappa_A} \ar[r]_-{\varphi}
& \Bdd(X)
.}
\]
\end{pbm}

As stated, this problem is very general and a complete answer seems out of reach.
However, in \autoref{sec:extRepr} we obtain a complete answer if we require the extension to be the `canonical' one, in the following sense.
In \cite{GarThi16arX:PredualsBXY}, we introduced a natural multiplication on $\Bdd(X,X^{**})$ and
a multiplicative operator $\alpha_X\colon(\Bdd(X)^{**},\rtArensProd)\to\Bdd(X,X^{**})$;
see \autoref{pgr:extReprXXdd} for details.

The composition $\alpha_X\circ\varphi^{**}\colon A^{**}\to \Bdd(X,X^{**})$ is a canonical candidate for a solution to \autoref{pbm:extendRepr}; the only
condition that may fail is that its range is not always contained in $\Bdd(X)$.
Besides being natural, this map has the remarkable feature that, whenever it is a solution to the extension problem, then it is in fact multiplicative with respect to \emph{both} Arens products.
Other particular aspects of this canonical solution are exploited in \autoref{sec:projEssSp}; see also Theorem~\ref{thmx:compl} below.

In the main result of this work, we characterize when a canonical extension exists:

\begin{thmx}\label{thmx:ext}
Let $\varphi\colon A\to\Bdd(X)$ be a representation.
Then \autoref{pbm:extendRepr} has a canonical solution (for either Arens product, or both)
if and only if the orbit map $\mathrm{ev}_x\circ\varphi\colon A\to X$ is weakly compact for all $x\in X$.
\end{thmx}

The easiest situation to apply \autoref{prp:charExtRepr} is when \emph{every} operator $A\to X$ is weakly compact.
In this case, \emph{every} representation of $A$ on $X$ has a canonical extension to a representation of $A^{**}$ (endowed with either Arens product) on $X$.
There are many interesting situations in which every operator $A\to X$ is weakly compact:
If $A$ or $X$ is reflexive;
or if $A$ is a \ca{} and $X$ does not contain an isomorphic copy of $c_0$;
see \autoref{pgr:extWCpct}.

If $\varphi\colon A\to\Bdd(X)$ is a representation, we define its \emph{essential space} $X_\varphi$ as the closed linear span of $\{\varphi(a)x\colon a\in A, x\in X\}$, and
we say that $\varphi$ is nondegenerate if $X_\varphi=X$.
Nondegeneracy is a technically important condition, which grants one access to extension arguments involving multiplier algebras.
For a locally compact group $G$, nondegeneracy is precisely the assumption needed to be able to `disintegrate' a representation $L^1(G)\to \Bdd(X)$ to an action of $G$ on $X$.

In practice, however, one often encounters nondegenerate representation.
The difference is relatively insignificant for representations on Hilbert spaces, since one can simply restrict the attention to the essential space (which is again a Hilbert space).
For a more general Banach space $X$, the distinction is more important.
It is particularly valuable to know when $X_{\varphi}$ is complemented in $X$, since many properties of Banach spaces pass to complemented subspaces.
In \autoref{sec:projEssSp}, we use Theorem~\ref{thmx:ext} to show that this is often the case:

\begin{thmx}\label{thmx:compl}
Let $A$ be a Banach algebra with a bounded left approximate indentity, and let $\varphi\colon A\to\Bdd(X)$ be a representation.
If the orbit map $\mathrm{ev}_x\circ\varphi\colon A\to X$ is weakly compact for all $x\in X$, then $X_\varphi$ is complemented in $X$.
\end{thmx}

We make a few comments about the proof of Theorem~\ref{thmx:compl}.
When $A$ has a bounded left approximate identity, the Banach algebra $(A^{**},\rtArensProd)$ has a left unit $e$;
see \autoref{pgr:approxUnit}.
Thus, if the representation $\varphi\colon A\to\Bdd(X)$ has an extension to some representation $\tilde{\varphi}\colon (A^{**},\rtArensProd)\to\Bdd(X)$,
then $\tilde{\varphi}(e)$ is a natural candidate for a projection onto the essential space.
In general, the range of the projection $\tilde{\varphi}(e)$ can be strictly bigger than the essential space of $\varphi$.
Nevertheless, and somewhat surprisingly, we show that this is not the case for the canonical extensions from Theorem~\ref{thmx:ext}.
We also obtain bounds on the norm of the projection, which allows us to show that the essential space is $1$-complemented if $\varphi$ is contractive and $A$ has a contractive left approximate identity;
see \autoref{prp:essSpComplRefl}.

In \autoref{sec:reprCalgLp} and \autoref{sec:univRepr} we present two applications of our results to what are arguably the most relevant classes of Banach algebras: $C^*$-algebras and group algebras.
Concretely, we show in \autoref{prp:CalgLpOpAlg} that a \ca{} can be isometrically represented on some $L^p$-space, for $p\in [1,\infty)\setminus\{2\}$, if and only if it is commutative.
This result is to some extent unexpected, and it should be compared with the fact that every \ca{} can be represented on a \emph{noncommutative} $L^p$-space, for any $p\in [1,\infty)$;
see \autoref{rmk:CalgLpOpAlg}.

Let $G$ be a locally compact group.
The completion of $L^1(G)$ for nondegenerate representations on $L^p$-spaces is called the \emph{universal group $L^p$-operator algebra}, denoted $F^p(G)$; see \cite{Phi13arX:LpCrProd, GarThi15GpAlgLp}.
The Banach algebra $F^p(G)$ captures the isometric representation theory of $G$ on $L^p$-spaces, and is universal with respect to nondegenerate representations of $L^1(G)$ on $L^p$-spaces.
Using Theorem~\ref{thmx:compl}, we show that $F^p(G)$ is universal for \emph{all} representations of $L^1(G)$ on $L^p$-spaces; see \autoref{prp:degReprFpG}.
The results of this paper, in particular \autoref{prp:CalgLpOpAlg} and \autoref{prp:degReprFpG}, have been applied in \cite{GarThi16arX:ReprConvLq} to give a complete answer to the question of when
$F^p(G)$ and $F^q(G)$ are isometrically isomorphic, among others. 

\subsection*{Acknowledgements}

The authors would like to thank Philip G.~Spain for valuable electronic correspondence.
The authors would like to thank the staff and organizers, and S{\o}ren Eilers in particular, for the hospitality during their visits to the Research Program \emph{Classification of operator algebras, complexity, rigidity and dynamics}, held at the Institut Mittag-Leffler, between January and April of 2016.



\section{Extending representations of a Banach algebra to its bidual}
\label{sec:extRepr}

Throughout this section, $A$ denotes a Banach algebra, $X$ denotes a Banach space, and $\varphi\colon A\to\Bdd(X)$ denotes a representation.
Using the multiplicative operator $\alpha_X\colon(\Bdd(X)^{**},\rtArensProd)\to\Bdd(X,X^{**})$ constructed in \cite{GarThi16arX:PredualsBXY}, we extend $\varphi$ to a multiplicative operator $\tilde{\varphi}\colon (A^{**},\rtArensProd)\to\Bdd(X,X^{**})$;
see \autoref{pgr:extReprXXdd} and \autoref{prp:extReprXXdd}.
The main result of this section is \autoref{prp:charExtRepr}, where we characterize when the image of $\tilde\varphi$ is contained in $\Bdd(X)$ in terms of weak compactness of orbit maps $A\to X$.\
If this is the case, it follows that $\widetilde{\varphi}$ is multiplicative not only with respect to $\rtArensProd$, but also with respect to $\ltArensProd$.


\begin{pgr}
\label{pgr:extReprXXdd}
Given operators $a,b\colon X\to X^{**}$, their product $ab$ is defined as the composition $\kappa_{X^*}^*\circ a^{**}\circ b$.
This gives $\Bdd(X,X^{**})$ the structure of a unital Banach algebra, with unit $\kappa_X$;
see \cite[Paragraph~3.9]{GarThi16arX:PredualsBXY}.
Note that $\Bdd(X,X^{**})$ is \emph{anti-isomorphic} to $\Bdd(X^*)$ by \cite[Proposition~3.11]{GarThi16arX:PredualsBXY}

We let $\gamma_X\colon\Bdd(X)\to\Bdd(X,X^{**})$ be given by $\gamma_X(a)=\kappa_X\circ a$ for $a\in\Bdd(X)$.
Then $\gamma_X$ is an isometric, multiplicative operator;
see \cite[Paragraph~3.11]{GarThi16arX:PredualsBXY}.

Let $\alpha_X\colon\Bdd(X)^{**}\to\Bdd(X,X^{**})$ be the contractive operator introduced in \cite[Definition~3.16]{GarThi16arX:PredualsBXY}.
Instead of recalling the definition of $\alpha_X$, we provide a formula to compute $\alpha_X(S)$ for $S\in\Bdd(X)^{**}$;
see \autoref{prp:charAlpha}.
(For the purposes of this paper, \autoref{prp:charAlpha} can be taken as the definition of $\alpha_X$.)

In general, the bitranspose of a multiplicative operator between Banach algebras is again multiplicative if both bidual Banach algebras are equipped with the same Arens product.
Thus, $\varphi^{**}\colon (A^{**},\rtArensProd)\to(\Bdd(X)^{**},\rtArensProd)$ is multiplicative.
The map $\alpha_X$ is multiplicative for the right Arens product $\rtArensProd$ on $\Bdd(X)^{**}$, and we have $\gamma_X=\alpha_X\circ\kappa_{\Bdd(X)}$;
see Corollary~3.23 and Lemma~3.19 in \cite{GarThi16arX:PredualsBXY}.
We thus obtain the following commutative diagram, where each map is a multiplicative operator:
\[
\xymatrix{
(A^{**},\rtArensProd) \ar[r]^-{\varphi^{**}}
& (\Bdd(X)^{**},\rtArensProd) \ar[r]^{\alpha_X}
& \Bdd(X,X^{**}). \\
A \ar@{^{(}->}[u]^{\kappa_{A}} \ar[r]_-{\varphi}
& \Bdd(X) \ar@{^{(}->}[u]^{\kappa_{\Bdd(X)}} \ar[ur]_{\gamma_X}
}
\]
\end{pgr}


\begin{prp}
\label{prp:extReprXXdd}
The map $\tilde{\varphi}=\alpha_X\circ\varphi^{**}\colon A^{**}\to\Bdd(X,X^{**})$ is multiplicative for the right Arens product $\rtArensProd$ on $A^{**}$.
Moreover, we have $\|\tilde{\varphi}\|=\|\varphi\|$.
\end{prp}
\begin{proof}
It remains to verify that $\|\tilde{\varphi}\|=\|\varphi\|$.
By \cite[Remark~3.17]{GarThi16arX:PredualsBXY}, the map $\alpha_X$ is contractive, so $\|\tilde{\varphi}\|\leq\|\varphi\|$.
The converse estimate follows using that $\gamma_X\circ\varphi = \tilde{\varphi}\circ\kappa_A$ and that $\gamma_X$ is isometric.
\end{proof}

Recall that an operator $f\colon X\to Y$ is said to be \emph{weakly compact} if the image of the unit ball of $X$ is contained in a weakly compact subset of $Y$.
By Gantmacher's theorem, $f$ is weakly compact if and only if its transpose $f^*$ is.
Moreover, $f$ is weakly compact if and only if the image of $f^{**}\colon X^{**}\to Y^{**}$ is contained in $Y$;
see \cite[Theorem~VI.5.5, p.185]{Con90Book:FctlAna}.

\begin{lma}
\label{prp:wkCpctSubsp}
Let $X$ and $Y$ be Banach spaces, let $Y_0\subseteq Y$ be a closed subspace, and let $f\colon X\to Y$ be an operator whose image is contained in $Y_0$.
Then $f$ is weakly compact if and only if the image of $f^{**}\colon X^{**}\to Y^{**}$ is contained in $Y_0$.
\end{lma}
\begin{proof}
If the image of $f^{**}$ is contained in $Y_0$, then it is also contained in $Y$, and consequently $f$ is weakly compact.
To show the backward implication, assume that $f$ is weakly compact.
Let $\iota\colon Y_0\to Y$ be the inclusion map.
Using that the image of $f$ is contained in $Y_0$, we let $f_0\colon X\to Y_0$ be the unique map such that $f=\iota\circ f_0$.

We identify $Y_0^{**}$ with a subspace of $Y^{**}$, and $Y$ with a subspace of $Y^{**}$.
There is a commutative diagram:
\[
\xymatrix{
X^{**} \ar[r]^{f_0^{**}}
& Y_0^{**} \ar@{^{(}->}[r]^{\iota^{**}}
& Y^{**} \\
X \ar[r]_{f_0} \ar@{^{(}->}[u]^{\kappa_X}
& Y_0 \ar@{^{(}->}[r]_{\iota} \ar@{^{(}->}[u]^{\kappa_{Y_0}}
& Y. \ar@{^{(}->}[u]^{\kappa_Y} \\
}
\]

We claim that $Y_0 = Y_0^{**} \cap Y$ in $Y^{**}$.
The inclusion `$\subseteq$' is clear.
To show the reverse inclusion, let $y\in Y_0^{**}\cap Y$.
Choose a net $(y_j)_{j\in J}$ in $Y_0$ that converges \weakStar{} to $y$ in $Y_0^{**}$.
Since the inclusion $\iota^{**}$ is \weakStar{} continuous, we deduce that $(y_j)_{j\in J}$ converges \weakStar{} to $y$ in $Y^{**}$.
Since the net $(y_j)_{j\in J}$ and $y$ belong to $Y$, it follows that $(y_j)_{j\in J}$ converges weakly to $y$.
It follows from the Hahn-Banach theorem that (norm-)closed subspaces are weakly closed, which implies that $y$ belongs to $Y_0$, which proves the claim.

Since $f$ is weakly compact, the image of $f^{**}$ is contained in $Y$.
Using that $\iota^{**}$ is isometric, it follows that the image of $f^{**}$ is contained in the image of $f_0^{**}$, which in turn is contained in $Y_0^{**}$.
Using the claim, the result follows.
\end{proof}

The following is the case $X=Y$ of \cite[Lemma~3.18]{GarThi16arX:PredualsBXY}.

\begin{lma}
\label{prp:charAlpha}
For $x\in X$, let $\ev_x\colon\Bdd(X)\to X$ be the evaluation map. 
Then $\alpha_X(S)(x) = \ev_x^{**}(S)$ for all $S\in\Bdd(X)^{**}$ and $x\in X$.
\end{lma}

We consider $\Bdd(X)$ as a subalgebra of $\Bdd(X,X^{**})$ via the map $\gamma_X$.
The following consequence is our desired (partial) answer to \autoref{pbm:extendRepr}.
\begin{thm}
\label{prp:charExtRepr}
Let $\varphi\colon A\to\Bdd(X)$ be a representation of a Banach algebra $A$ on a Banach space $X$, and set $\tilde{\varphi}=\alpha_X\circ\varphi^{**}\colon A^{**}\to\Bdd(X,X^{**})$.
Then the image of $\tilde{\varphi}$ is contained in $\Bdd(X)$ if and only if for every $x\in X$, the orbit map $\ev_x\circ\varphi\colon A\to X$ is weakly compact.
Moreover, if this is the case, then:
\begin{enumerate}
 \item The essential spaces of $\varphi$ and $\tilde\varphi$ agree.
 \item $\widetilde{\varphi}\colon A^{**}\to \Bdd(X)$ is multiplicative with respect to $\rtArensProd$ and $\ltArensProd$.
 \item $\widetilde{\varphi}$ is an extension of $\varphi$.
\end{enumerate}
\end{thm}
\begin{proof}
Using \autoref{prp:charAlpha} at the first step, we have
\[
\tilde\varphi(S)(x) = \ev_x^{**}(\varphi^{**}(S)) = (\ev_x\circ\varphi)^{**}(S),
\]
for all $S\in A^{**}$ and $x\in X$.
Thus, given $x\in X$, the orbit map $\ev_x\circ\varphi$ is weakly compact if and only if $\tilde\varphi(S)(x)$ belongs to $\Bdd(X)$ for every $S\in A^{**}$,
by the comments before \autoref{prp:wkCpctSubsp}.
This implies the first statement.

From now on, we assume that every orbit map $\ev_x\circ\varphi$ is weakly compact, and hence that the image of $\widetilde{\varphi}$ is contained
in the image of $\gamma_X$.

(1). We clearly have $X_\varphi\subseteq X_{\tilde\varphi}$.
Conversely, given $x\in X$, since $\ev_x\circ\varphi$ is weakly compact with image contained in $X_\varphi$, it follows from \autoref{prp:wkCpctSubsp} that the image of $(\ev_x\circ\varphi)^{**}$ is contained in $X_\varphi$ as well.
Using \autoref{prp:charAlpha} at the first step, we deduce that
\[
\tilde\varphi(S)x = (\ev_x\circ\varphi)^{**}(S) \in X_\varphi,
\]
for every $S\in A^{**}$ and $x\in X$.
Hence, $X_{\tilde\varphi}\subseteq X_\varphi$, as desired.

(2). That $\widetilde{\varphi}$ is multiplicative with respect to $\rtArensProd$ follows from \autoref{prp:extReprXXdd}.
Given $a,b\in A$, we use the second identity of Lemma~3.20 of \cite{GarThi16arX:PredualsBXY} at the second step (with $T=\varphi^{**}(a)$ and $S=\widetilde{\varphi}(b)\in \gamma_X(\Bdd(X))$),
and Lemma~3.21 of \cite{GarThi16arX:PredualsBXY} at the third step, to get
\[\widetilde{\varphi}(a)\widetilde{\varphi}(b)= \alpha_X(\varphi^{**}(a))\alpha_X(\varphi^{**}(b))=\varphi^{**}(a)\alpha_X(\varphi^{**}(b))=\alpha_X(\varphi^{\ast\ast}(a\ltArensProd b)),\]
as desired.

(3). This is immediate, since $\alpha_X\circ\kappa_{\Bdd(X)}=\gamma_X$.
\end{proof}




\begin{pgr}
\label{pgr:extWCpct}
If $X$ and $Y$ are Banach spaces, there are some very general situations in which every operator $X\to Y$ is automatically weakly compact.
For example, if either $X$ or $Y$ is reflexive (since a Banach space $E$ is reflexive if and only if $\id_E$ is weakly compact).
However, there are situations where neither $X$ nor $Y$ is reflexive, yet every operator $X\to Y$ is weakly compact.
This is for instance the case when $X$ has \Pelc's property $(V)$ and $Y$ contains no isomorphic copy of $c_0$;
this is essentially an immediate consequence of \cite[Theorem~5]{BesPel58BasesUncondConvSeries}.



By a deep result of Pfitzner, \cite[Carollary~6]{Pfi94WkCpctDualCAlg}, every \ca{} $A$ has property $(V)$.
It follows that if $Y$ contains no isomorphic copy of $c_0$ then every operator $A\to Y$ is weakly compact.
This had previously been obtained in \cite[Theorem~4.2]{AkeDodGam72WkCpctDualCAlg}.

The class of spaces that do not contain an isomorphic copy of $c_0$ is known to 
include all reflexive space, and all $L^p$-spaces for $p\in[1,\infty)$.
\end{pgr}

Since \ca{s} are Arens regular, we do not have to specify the choice of Arens product in their biduals.

\begin{cor}
\label{prp:extReprCa}
Let $A$ be a \ca{} and let $X$ be a Banach space that does not contain an isomorphic copy of $c_0$.
Then every representation of $A$ on $X$ has a canonical extension to a representation of $A^{**}$.
\end{cor}

For unital representations of unital \ca{s}, \autoref{prp:extReprCa} has implicitly been obtained by Spain, \cite[Section~9]{Spa15ReprCaInDualBAlg}.
However, for our applications in \autoref{sec:projEssSp}, we have to consider nonunital Banach algebras and nonunital representations.

\section{Projections onto the essential space of a representation}
\label{sec:projEssSp}

Let $X$ be a Banach space, and let $X_0\subseteq X$ be a closed subspace.
Recall that a \emph{projection} from $X$ onto $X_0$ is an idempotent operator on $X$ with image $X_0$.
For $\lambda\geq 1$, we say that $X_0$ is \emph{$\lambda$-complemented} in $X$ if there exists a projection from $X$ onto $X_0$ of norm at most $\lambda$.
The space $X_0$ is said to be \emph{complemented} if it is $\lambda$-complemented for some $\lambda$.
In this section, we study the following:

\begin{pbm}
\label{pbm:essSpace}
Let $A$ be a Banach algebra, and let $X$ be a Banach space.
Under which assumptions on $A$ and $X$ does it follow that the essential subspace of every representation $A\to\Bdd(X)$ is complemented?
Moreover, if such a projection exists, can its norm be estimated?
\end{pbm}


Of course, any solution to \autoref{pbm:essSpace} depends crucially on $A$ and $X$.
The better behaved $A$ is, the fewer assumptions we need to put on $X$, and conversely.
The extreme points for this tradeoff are given by the case that $X$ is isomorphic to a Hilbert space (where every subspace is 1-complemented), and on the other end that $A$ has a left unit
(where the image of the left unit is the desired projection).


The main result of this section is \autoref{prp:essSpCompl}, which provides a positive solution to \autoref{pbm:essSpace} for the case that $A$ has a bounded left approximate identity and every operator $A\to X$ is weakly compact.
The following example shows that every Banach space which is not a Hilbert space admits a representation of \emph{some} Banach algebra whose essential space is not complemented.

\begin{exa}
\label{exa:XnotHilbert}
Let $X$ be a Banach space not isomorphic to a Hilbert space.
By the Lindenstrauss-Tzafriri theorem, \cite{LinTza71ComplSubspPbm}, we can choose a closed subspace $X_0$ that is not complemented in $X$.
Consider $\Bdd(X,X_0)$, which we identify with the space of operators $a\in\Bdd(X)$ satisfying $a(X)\subseteq X_0$.
Then $\Bdd(X,X_0)$ is a closed subalgebra of $\Bdd(X)$.
Then the inclusion map is an isometric representation $\varphi\colon\Bdd(X,X_0)\to\Bdd(X)$, whose
essential space is easily seen to be exactly $X_0$.
\end{exa}

Analogously, one may ask what happens if the Banach algebra is fixed but the Banach space is allowed to vary:

\begin{qst}
\label{qst:AnoLeftUnit}
Let $A$ be a Banach algebra without left unit.
Does there exist a representation $\varphi\colon A\to\Bdd(X)$ on some Banach space $X$ such that the essential space of $\varphi$ is not complemented?
\end{qst}

\begin{exa}
\label{exa:c0_on_linfty}
Let $\varphi\colon c_0\to \Bdd(\ell_\infty)$ be the canonical isometric representation given by pointwise multiplication.
It is easy to show that the essential space of $\varphi$ is exactly the subspace $c_0$ in $\ell_\infty$.
However, Phillips' theorem states that $c_0$ is not complemented in $\ell_\infty$.
(See \cite{Whi66ProjMC0} for a simple proof of Phillips' theorem.)
This provides a positive answer to \autoref{qst:AnoLeftUnit} for $A=c_0$.
\end{exa}

\begin{pgr}
\label{exa:A_on_MA}
\label{pgr:A_on_MA}
We can generalize \autoref{exa:c0_on_linfty} to many other nonunital \ca{s}, as follows.
Taylor showed in \cite[Corollary~3.7]{Tay72GenPhillipsThm} that a nonunital $C^*$-algebra $A$ is uncomplemented in its multiplier algebra $M(A)$ if $A$ has a well-behaved approximate identity.
(See \cite{Tay72GenPhillipsThm} for the definition.)
By \cite[Proposition~3.1]{Tay72GenPhillipsThm}, it follows that 
every non-unital, $\sigma$-unital \ca{} is uncomplemented in its multiplier algebra.
In particular, this gives a positive answer to \autoref{qst:AnoLeftUnit} for every separable \ca{}.

In general, however, a \ca{} need not be uncomplemented in its multiplier algebra.
Indeed, there exist non-unital \ca{s} whose multiplier algebras agree with their minimal unitizations, and every \ca{} is complemented in its minimal unitization.
For examples and a further discussion, 
we refer to \cite{GhaKos16arX:ExtCpctOpsTrivMult}.
\end{pgr}

\begin{pgr}
\label{pgr:approxUnit}
Let $A$ be a Banach algebra.
Given a bounded left approximate identity $(e_j)_{j\in J}$ in $A$, let $e$ be any \weakStar{} cluster point of the net $(\kappa_A(e_j))_{j\in J}$ in $A^{**}$.
Then $e\rtArensProd a=a$ for every $a\in A$.
Since the right Arens product is \weakStar{} continuous in the second variable, we deduce that $e\rtArensProd S=S$ for all $S\in A^{**}$.
Thus, $e$ is a left unit for the right Arens products.
The converse also holds, and hence $A$ has a bounded left approximate identity (with bound $\lambda$) if and only if $(A^{**},\rtArensProd)$ has a left unit (with norm at most $\lambda$);
see \cite[Proposition~5.1.8, p.527]{Pal94BAlg1}.
\end{pgr}

\begin{lma}
\label{prp:projIntoBidual}
Let $A$ be a Banach algebra, 
let $X$ be a Banach space, and let $\varphi\colon A\to\Bdd(X)$ be a representation.
Set $\tilde{\varphi}=\alpha_X\circ\varphi^{**}\colon A^{**}\to\Bdd(X,X^{**})$. 
Let $e\in A^{**}$ be a left unit for $(A^{**},\rtArensProd)$, 
and set $p=\tilde{\varphi}(e)$.
Then $p$ is an idempotent in $\Bdd(X,X^{**})$ with $\|p\|\leq\|e\|\|\varphi\|$ and such that $p\tilde{\varphi}(S)=\tilde{\varphi}(S)$ for all $S\in A^{**}$.
\end{lma}
\begin{proof}
By \autoref{prp:extReprXXdd}, we have $\|\tilde\varphi\|=\|\varphi\|$, which implies that
\[
\|p\|
= \| \tilde{\varphi}(e) \|
\leq \|\tilde{\varphi}\| \|e\|
= \|\varphi\| \|e\|.
\]

Since $\tilde{\varphi}$ is multiplicative for the right Arens product by \autoref{prp:extReprXXdd}, we deduce that
$p^2
= \tilde{\varphi}(e\rtArensProd e)
= \tilde{\varphi}(e)
= p$.
Finally, for every $S\in A^{**}$, we have
\[
p\tilde{\varphi}(S)
= \tilde{\varphi}(e) \tilde{\varphi}(S)
= \tilde{\varphi}(e\rtArensProd S)
= \tilde{\varphi}(S).\qedhere
\]
\end{proof}

The following is a (partial) solution to \autoref{pbm:essSpace}.

\begin{thm}
\label{prp:essSpCompl}
Let $A$ be a Banach algebra with a bounded left approximate identity with bound $\lambda$, let $X$ be a Banach space such that every operator $A\to X$ is weakly compact, and let $\varphi\colon A\to\Bdd(X)$ be a representation.
Then the essential space of $\varphi$ is $\lambda\|\varphi\|$-complemented in $X$.
\end{thm}
\begin{proof}
Set $\tilde{\varphi}=\alpha_X\circ\varphi^{**}\colon A^{**}\to\Bdd(X,X^{**})$, as in \autoref{prp:extReprXXdd}.
It follows from \autoref{prp:charExtRepr} that the image of $\tilde{\varphi}$ is contained in $\Bdd(X)$, and that $X_{\varphi}=X_{\widetilde{\varphi}}$.
As noted in \autoref{pgr:approxUnit}, there exists a left unit $e$ for $(A^{**},\rtArensProd)$ with $\|e\|\leq\lambda$.
Set $p=\tilde{\varphi}(e)\in\Bdd(X)$.
Then $p$ is an idempotent and $\|p\|\leq\lambda\|\varphi\|$ by \autoref{prp:projIntoBidual}.
The range of $p$ is contained in the essential space of $\tilde{\varphi}$.
Conversely, it follows from \autoref{prp:projIntoBidual} that $p$ acts as the identity on the essential space of $\tilde{\varphi}$.
Thus, $p$ is the desired projection onto the essential space of $\varphi$.
\end{proof}

\begin{cor}
\label{prp:essSpComplRefl}
Let $A$ be a Banach algebra with a contractive left approximate identity, let $X$ be a reflexive Banach space, and let $\varphi\colon A\to\Bdd(X)$ be a contractive representation.
Then the essential space of $\varphi$ is $1$-complemented in $X$.
\end{cor}
\begin{proof}
Since $X$ is reflexive, every operator $A\to X$ is weakly compact.
Thus, we may apply \autoref{prp:essSpCompl}.
\end{proof}

\begin{cor}
\label{prp:essSpReprCa}
Let $A$ be a \ca{}, let $X$ be a Banach space that does not contain an isomorphic copy of $c_0$, and let $\varphi\colon A\to\Bdd(X)$ be a contractive representation.
Then the essential space of $\varphi$ is $1$-complemented in $X$.
\end{cor}
\begin{proof}
The assumptions imply that every operator $A\to X$ is weakly compact;
see \autoref{pgr:extWCpct}.
Moreover, every \ca{} has a contractive approximate identity.
Thus, we may apply \autoref{prp:essSpComplRefl} to deduce the statement.
\end{proof}

\section{Representations of multiplier algebras}
\label{sec:reprMult}

Throughout this section, we will use the notion of (left) multiplier algebras; see, for example, \cite{Joh64IntroThyCentralizers} and \cite{Daw10Multipliers}.
It is well-known that a nondegenerate representation of a Banach algebra $A$ with bounded approximate identity can be extended to a representation of the (left) multiplier algebra $M_l(A)$.
For our purposes, we need to show that a nondegenerate \emph{isometric} representation of $A$ can be extended to an unital, \emph{isometric} representation of $M_l(A)$.
This result is certainly also well-known, but we could not locate a reference.
We therefore include an argument.

\begin{thm}
\label{prp:extendNondegRepr}
Let $X$ be a Banach space, let $A$ be a Banach algebra with a contractive, left approximate identity, and let $\varphi\colon A\to\Bdd(X)$ be a contractive, nondegenerate representation.
Then there is a unique contractive, unital representation $\widetilde{\varphi}\colon M_l(A)\to\Bdd(X)$ such that the following diagram commutes:
\[
\xymatrix{
A \ar@{^{(}->}[d] \ar[r]^-{\varphi}
& \Bdd(X). \\
M_l(A) \ar[ur]_{\widetilde{\varphi}}
}
\]
Moreover, for every $a\in A$, $x\in X$ and $T\in M_l(A)$, we have
\begin{align}
\label{prp:extendNondegRepr:Eq1}
\widetilde{\varphi}(T)(\varphi(a)(x))
= \varphi(T(a))(x).
\end{align}
Furthermore, if $\varphi$ is isometric, then so is $\widetilde{\varphi}$.
\end{thm}
\begin{proof}
Let $T\in M_l(A)$. We claim that \eqref{prp:extendNondegRepr:Eq1} defines a map $\widetilde{\varphi}(T)\colon X\to X$.
By the Cohen-Hewitt factorization theorem, we have $X=\varphi(A)X$.
To show that $\widetilde{\varphi}(T)$ is well-defined let $a,b\in A$ and $x,y\in X$ satisfy $\varphi(a)(x)=\varphi(b)(y)$.
Let $(e_j)_{j\in J}$ be a contractive, left approximate identity in $A$.
Using that $T$ is continuous at the first and last step, and that $T$ is a left multiplier at the second step, we obtain that
\begin{align*}
\varphi(T(a))(x)
=\lim_{j\in J} \varphi(T(e_ja))(x)
=\lim_{j\in J} \varphi(T(e_j)a)(x)
=\lim_{j\in J} \varphi(T(e_j))(\varphi(a)(x)).
\end{align*}
Analogously, we have
\[\varphi(T(b))(y)=\lim_{j\in J} \varphi(T(e_j))(\varphi(b)(y)).\]
Thus $\varphi(T(a))(x)=\varphi(T(b))(y)$, and $\widetilde{T}$ is well-defined.

We claim that $\widetilde{\varphi}$ is a unital representation satisfying $\widetilde{\varphi}\circ L=\varphi$.
Given $S,T\in M_l(A)$, $a\in A$ and $x\in X$, we have
\[
\widetilde{\varphi}(ST)(\varphi(a)(x))
= \varphi(ST(\varphi(a)))(x)
= \widetilde{\varphi}(S)(\varphi(T(\varphi(a)))(x))
= \widetilde{\varphi}(S)\widetilde{\varphi}(T)(\varphi(a)(x)),
\]
and thus $\widetilde{\varphi}(ST)=\widetilde{\varphi}(S)\widetilde{\varphi}(T)$.
Similarly one shows that $\widetilde{\varphi}(\id_A)=\id_E$.
Further, 
\[
\widetilde{\varphi}(L_a)(\varphi(b)(x))
= \varphi(L_a(b))(x)
= \varphi(ab)(x)
= \varphi(a)(\varphi(b)(x)),
\]
for all $b\in A$, which implies that $\widetilde{\varphi}(L_a)=\varphi(a)$.
Thus, $\widetilde{\varphi}\circ L=\varphi$, as desired.

We claim that $\widetilde{\varphi}$ is conctractive. Given $T\in M_l(A)$,
we use the Cohen-Hewitt factorization theorem at the second step, to get
\begin{align*}
\|\widetilde{\varphi}(T)\|
&= \sup\big\{ \|\widetilde{\varphi}(T)(x)\| : x\in X, \|x\|<1 \big\} \\
&= \sup\big\{ \|\widetilde{\varphi}(T)(\varphi(a)(x))\| : a\in A, x\in X, \|a\|\leq 1, \|x\|<1 \big\} \\
&= \sup\big\{ \|\varphi(T(a))(x)\| : a\in A, x\in X, \|a\|\leq 1, \|x\|<1 \big\}\\
& \leq\ \|T\|,
\end{align*}
as desired. Finally, when $\varphi$ is isometric, then $\widetilde{\varphi}$ is also isometric:
\begin{align*}
\|T\|
&= \sup\big\{ \|T(a)(x)\| : a\in A, x\in X, \|a\|\leq 1, \|x\|\leq 1 \big\} \\
&= \sup\big\{ \|\varphi(T(a)(x))\| : a\in A, x\in X, \|a\|\leq 1, \|x\|\leq 1 \big\} \\
&= \sup\big\{ \|\widetilde{\varphi}(T)(\varphi(a)x)\| : a\in A, x\in X, \|a\|\leq 1, \|x\|\leq 1 \big\}\\
&\leq \|\widetilde{\varphi}(T)\|.\qedhere
\end{align*}
\end{proof}

\begin{rmk}
\label{rmk:extendNondegRepr}
If $A$ has a contractive, approximate identity, then the analog of \autoref{prp:extendNondegRepr} holds with the multiplier algebra in place of the left multiplier algebra, since in
this case $M(A)$ can be identified with a closed subalgebra of $M_l(A)$.
\end{rmk}

Let $p\in[1,\infty)$.
Recall that 
a Banach space is called an $SL^p$-space ($QL^p$-space, $QSL^p$-space) if it is isometrically isomorphic to a subspace (respectively, a quotient, a quotient of a subspace) of an $L^p$-space.
In \autoref{prp:essSpLp}, we consider classes of reflexive Banach spaces that are closed under passing to $1$-complemented subspaces.
By \cite[Theorem~6]{Tza69CntrProjLp}, every $1$-complemented subspace of an $L^p$-space is again an $L^p$-space.
It follows that for each $p$, the class of $L^p$-spaces (of $QL^p$-spaces, of $SL^p$-spaces, of $QSL^p$-spaces) is closed under passing to $1$-complemented subspaces.
For an overview of classes of Banach spaces with this property, see \cite{NeaRus11ExContrProjPredualJBW}.

Let $\mathcal{E}$ be a class of Banach spaces.
A Banach algebra $A$ is called an \emph{$\mathcal{E}$-operator algebra} if there exists an isometric representation of $A$ on some Banach space in $\mathcal{E}$.
For instance, an $L^p$-operator algebra is a Banach algebra that admits an isometric representation on some $L^p$-space.
Such algebras have been studied in \cite{Phi13arX:LpAnalogsCtz}, \cite{Phi13arX:LpCrProd}, \cite{GarThi15BAlgInvIsoLp}, \cite{GarThi15GpAlgLp}, \cite{GarThi16QuotBAlgLp}.

\begin{thm}
\label{prp:essSpLp}
Let $\mathcal{E}$ be a class of reflexive Banach spaces that is closed under passing to $1$-complemented subspaces,
and let $A$ be an $\mathcal{E}$-operator algebra with a contractive left approximate identity.
Then there is a nondegenerate isometric representation of $A$ on a space in $\mathcal{E}$.
It follows that the left multiplier algebra of $A$ has a unital, isometric representation on a space in $\mathcal{E}$.

If $A$ has a contractive approximate identity, then the same holds for $M(A)$.
\end{thm}
\begin{proof}
Choose $X\in \mathcal{E}$ and an isometric representation $\varphi\colon A\to\Bdd(X)$.
Let $X_\varphi$ denote the essential space of $\varphi$.
By \autoref{prp:essSpComplRefl}, $X_\varphi$ is $1$-complemented in $X$, so $X_\varphi\in \mathcal{E}$.
The restriction of $\varphi$ to $X_\varphi$ is the desired nondegenerate representation.
Applying \autoref{prp:extendNondegRepr}, we obtain a unital, isometric representation of $M_l(A)$ on $X_\varphi$.
If $A$ has a contractive approximate identity, then $M(A)$ can be identified with a unital, closed subalgebra of $M_l(A)$, which implies the statements for $M(A)$.
\end{proof}

\begin{rmk}
If $A$ is a \ca{}, then \autoref{prp:essSpLp} also holds for the class of $L^1$-spaces;
see \autoref{prp:essSpReprCaLp}.
\end{rmk}

\section{Representations of \texorpdfstring{$C^*$}{C*}-algebras on \texorpdfstring{$L^p$}{Lp}-spaces}
\label{sec:reprCalgLp}

The goal of this section is to show that a \ca{} is necessarily commutative if it has an isometric representation on an $L^p$-spaces for $p\in [1,\infty)\setminus\{2\}$;
see \autoref{prp:CalgLpOpAlg}.
This result is a crucial ingredient to obtain the results in \cite{GarThi16arX:ReprConvLq}, where it is shown that algebras of convolution operators on $L^p(G)$, for a nontrivial locally compact group $G$, are not presentable on an $L^q$-space unless $p=q$, or $p$ and $q$ are conjugate H\"{o}lder exponents, or $p=2$ and $G$ is commutative.

To main ingredients to prove \autoref{prp:CalgLpOpAlg} are:
first, a reduction to the unital case using the results in \autoref{sec:reprMult}; and second, Lamperti's theorem \cite{Lam58IsoLp}.

For $p>1$, the following result also follows from \autoref{prp:essSpLp}.

\begin{lma}
\label{prp:essSpReprCaLp}
Let $p\in[1,\infty)$, and let $A$ be a \ca{} with an isometric representation on an $L^p$-space.
Then the multiplier algebra of $A$ has a unital, isometric representation on an $L^p$-space.
\end{lma}
\begin{proof}
Choose an $L^p$-space $X$ and an isometric representation $\varphi\colon A\to\Bdd(X)$.
As observed in \autoref{pgr:extWCpct}, $X$ does not contain an isomorphic copy of $c_0$.
It follows from \autoref{prp:essSpReprCa}, that $X_\varphi$ is $1$-complemented in $X$.
We may then argue as in the proof of \autoref{prp:essSpLp} to deduce that $X_\varphi$ is an $L^p$-space and that the restriction of $\varphi$ to $X_\varphi$ extends
to a unital isometric representation of $M(A)$ on $X_\varphi$.
\end{proof}

Given a unital Banach algebra $A$, recall that $a\in A$ is said to be \emph{hermitian} if $\|\exp(ita)\|=1$ for every $t\in\RR$.
(Equivalently, $\|\exp(ita)\|\leq1$ for every $t\in\RR$.)
An element of a unital \ca{} is hermitian if and only if it is self-adjoint.

Let $A$ and $B$ be unital Banach algebras, let $\varphi\colon A\to B$ be a multiplicative operator, and let $a\in A$ be hermitian.
If $\varphi$ is unital and contractive, then $\varphi(a)$ is hermitian in $B$.
Indeed, using that $\varphi$ is unital, we have $\exp(it\varphi(a))=\varphi(\exp(ita))$ for every $t$.
Further, using that $\varphi$ is contractive, we obtain that
\[
\| \exp(it\varphi(a)) \| = \| \varphi(\exp(ita))\| \leq \| \exp(ita) \| =1,
\]
as desired.
If $\varphi$ is not unital or not contractive, then $\varphi(a)$ need not be hermitian.


\begin{lma}
\label{prp:hermOpLp}
Let $p\in [1,\infty)\setminus\{2\}$, let $(X,\mu)$ be a complete, $\sigma$-finite measure space, and let $a\in\Bdd(L^p(\mu))$.
Then $a$ is hermitian if and only if $a$ is the multiplication operator corresponding to some bounded, measurable function $X\to\RR$.
In particular, all hermitian operators on $L^p(\mu)$ commute.
\end{lma}
\begin{proof}
Assume that $a$ is given by multiplication with a bounded, measurable function $h\colon X\to\RR$.
We need to show that $\|\exp(ita)\|\leq 1$ for every $t\in\RR$.
Let $t\in\RR$.
Then the operator $\exp(ita)$ is given by multiplication with the function $\exp(ith)$.
For each $x\in X$, we have $\exp(ith)(x)\in S^1$, and therefore $\|\exp(ith)\|_\infty\leq 1$.
It follows that $\|\exp(ita)\|\leq 1$, as desired.

To show the converse implication, assume that $a$ is hermitian.
By rescaling, if necessary, we may assume that $\|a\|\leq\tfrac{\pi}{2}$.
Let $\Isom(L^p(\mu))$ denote the set of surjective isometric operators $L^p(\mu)\to L^p(\mu)$.
For $t\in\RR$, set $u_t=\exp(ita)$.
Since $a$ is hermitian, we have $\|u_t\|\leq 1$ for every $t$.
Moreover, we have $u_tu_{-t}=u_{-t}u_t=\id$, which implies that $u_t$ belongs to $\Isom(L^p(\mu))$ for every $t$.
Moreover, the resulting map $u\colon [0,1]\to\Isom(L^p(\mu))$ is norm-continuous.

By Lamperti's theorem (in the form given in Theorem~4.4 in~\cite{GarThi15BAlgInvIsoLp};
see also \cite{Lam58IsoLp} for the original statement), for every $v\in\Isom(L^p(\mu))$ there exists a measurable function $h_v\colon X\to S^1$ and a measure class preserving automorphism $T_v\colon X\to X$
(see Definition~4.1 in \cite{GarThi15BAlgInvIsoLp}) such that
\begin{eqnarray*}
(v\xi)(x) =h_v(x)\left(\frac{d(\mu\circ T_v^{-1})}{d\mu}(x)\right)^{\frac{1}{p}}\xi(T_v^{-1}(x))
\end{eqnarray*}
for all $\xi\in L^p(\mu)$ and $\mu$-almost every $x\in X$.
(The transformation $T_v$ is called the \emph{spatial realization} of $v$ in \cite{Phi13arX:LpAnalogsCtz}.)

We may apply Lamperti's theorem to $u_t$, for each $t\in[0,1]$.
By Lemma~6.22 in \cite{Phi13arX:LpAnalogsCtz}, the spatial realizations of $u_0$ and $u_1$ agree.
Since $u_0$ is the identity map on $L^p(\mu)$, its spatial realization is the identity on $X$.
Thus, there exists a measurable function $h\colon X\to S^1$ such that $u_1$ is given by multiplication with $h$.

Set $T= \{it : t\in  [-\tfrac{\pi}{2},\tfrac{\pi}{2}]\}$, and set $P= \{z\in S^1 : \re(z)\geq 0\}$.
The restriction of the exponential map to $T$ is a bijection onto $P$.
We let $\log\colon P\to T$ denote the inverse map, which is contained in a branch of the logarithm and is therefore analytic on a neighborhood of $P$.

Let $\sigma(a)$ denote the spectrum of $a$.
Since $\|a\|\leq\tfrac{\pi}{2}$, the spectrum of $ia$ is contained in $T$, and consequently $\sigma(u_1)\subseteq P$.
Thus, we may apply (analytic) functional calculus to $u_1$ and obtain that $ia=\log(u_1)$.
Recall that $u_1$ is given by multiplication with the function $h\colon X\to S^1$.
After changing $h$ on a null-set, we may assume that $h$ takes image in $P$.
It follows that $ia$ is given by multiplication by the function $\log(h)$, which takes values in $T$.
This implies that $a$ is given by multiplication by $-i\log(h)$, which is a bounded, measurable real-valued function, as desired.
\end{proof}

\begin{thm}
\label{prp:CalgLpOpAlg}
Let $A$ be a \ca{} and let $p\in [1,\infty)\setminus\{2\}$.
Then $A$ can be isometrically represented on an $L^p$-space if and only if $A$ is commutative.
\end{thm}
\begin{proof}
Assume that $A$ is commutative.
By the Gelfand representation theorem, there exists a locally compact Hausdorff space $X$ such that $A\cong C_0(X)$.
By choosing a suitable family of Borel probability measures on $X$, it is easy to construct a faithful representation of $C_0(X)$
on an $L^p$-space, via a direct sum of representations by multiplication operators. We omit the details.

To show the converse, assume that $A$ has an isometric representation on an $L^p$-space.
By \autoref{prp:essSpReprCaLp}, there is a unital, isometric representation of the multiplier algebra $M(A)$ on some $L^p$-space.
Since $A$ is commutative if (and only if) $M(A)$ is commutative, we can assume that $A$ and its representation are unital.
Since any element in a \ca{} is a linear combination of two self-adjoint elements, it is enough to show that self-adjoint elements in $A$ commute.

Let $a,b\in A_{\mathrm{sa}}$.
Since the unital sub-\ca{} generated by $a$ and $b$, is a separable unital \ca{}, we may assume that $A$ itself is separable.
In this case, by \cite[Proposition~1.25]{Phi13arX:LpCrProd}, there exist a separable $L^p$-space $E$ and a nondegenerate (and hence automatically unital) isometric representation $\varphi\colon A\to \Bdd(E)$.
Since every separable $L^p$-space is isometrically isomorphic to $L^p(\mu)$ for some complete, $\sigma$-finite measure space $\mu$, and since $\varphi(a)$ and $\varphi(b)$ are hermitian, we deduce from \autoref{prp:hermOpLp} that $\varphi(a)$ and $\varphi(b)$ commute.
Since $\varphi$ is injective, the result follows.
\end{proof}

One important ingredient in the proof of \autoref{prp:CalgLpOpAlg} was that all hermitian operators on $L^p(\mu)$ commute, for a complete, $\sigma$-finite measure space $\mu$ and $p\in[1,\infty)\setminus\{2\}$.
We remark that there are other Banach spaces $X$ such that all hermitian operators on $X$ commute.
There are even Banach spaces $X$ such that every hermitian operator on $X$ is a multiple of the identity operator;
see \cite[Section~3]{BerSou74HermOpsOnBSp}.
For such $X$, there is no isometric, nondegenerate representation of any nonzero \ca{} $A\neq\CC$ on $X$.

\begin{rmk}
\label{rmk:CalgLpOpAlg}
\autoref{prp:CalgLpOpAlg} should be contrasted with the fact that \emph{every} $C^*$-algebra can be isometrically represented on a noncommutative $L^p$-space
To prove this, it is enough to show it for $\Bdd(\mathcal{H})$ for  a Hilbert space $H$.
Given $p\in[1,\infty)$, consider the associated noncommutative $L^p$-space $\mathcal{S}_p(H)$ of Schatten-$p$ class operators.
Using that $\mathcal{S}_p(H)$ is a (nonclosed) ideal in $\Bdd(H)$, we may define $\varphi\colon\Bdd(H)\to\Bdd(\mathcal{S}_p(H))$ by $\varphi(a)b=ab$, for $a\in\Bdd(H)$ and $b\in\mathcal{S}_p(H)$.
It is well-known that $\varphi$ is an isometric representation.
\end{rmk}

\section{Universal completions for representations on \texorpdfstring{$L^p$}{Lp}-spaces}
\label{sec:univRepr}

\begin{pgr}
Let $A$ be a Banach algebra, and let $\mathcal{E}$ be a class of Banach spaces.
We let $\Rep_{\mathcal{E}}(A)$ denote the class of all non-degenerate, contractive representations $A\to\Bdd(X)$, for $X$ ranging over Banach spaces in $\mathcal{E}$.
Given $a\in A$, we set
\[
\|a\|_{\mathcal{E}} = \sup\big\{ \|\varphi(a)\| : \varphi\in\Rep_{\mathcal{E}}(A) \big\}.
\]
We let $F_\mathcal{E}(A)$ denote the Hausdorff completion of $A$ with respect to the semi-norm $\|\cdot\|_{\mathcal{E}}$. 
Multiplication on $A$ extends to $F_\mathcal{E}(A)$ and turns it into a Banach algebra.
\end{pgr}

In the next result, we show that in the definition of $F_{\mathcal{E}}(A)$, one can sometimes consider \emph{all} representations on $\mathcal{E}$, and not just the nondegenerate ones.

\begin{thm}
Let $\mathcal{E}$ be a class of reflexive Banach spaces that is closed under passing to $1$-complemented subspaces, 
and let $A$ be a Banach algebra with a contractive left approximate identity, and let $a\in A$.
Then
\[
\|a\|_{\mathcal{E}} = \sup\big\{ \|\varphi(a)\| : \varphi\colon A\to\Bdd(X) \text{ contractive representation}, X\in\mathcal{E} \big\}.
\]
\end{thm}
\begin{proof}
Let $K$ denote the supremum of $\|\varphi(a)\|$ for all contractive representations $\varphi\colon A\to\Bdd(X)$ with $X\in\mathcal{E}$.
We clearly have $\|a\|_{\mathcal{E}}\leq K$.
To show the converse inequality, let $X\in\mathcal{E}$ and let $\varphi\colon A\to\Bdd(X)$ be a contractive representation.
We need to verify that $\|\varphi(a)\|\leq \|a\|_{\mathcal{E}}$.
It follows from \autoref{prp:essSpComplRefl} that $X_\varphi$ is $1$-complemented in $X$, and hence $X_\varphi$ belongs to $\mathcal{E}$.
Let $\varphi_0$ denote the restriction of $\varphi$ to $X_\varphi$.
Note that $\|\varphi(b)\|=\|\varphi_0(b)\|$ for every $b\in A$.
Thus, $\varphi_0$ is a nondegenerate, contractive representation of $A$ on a space in $\mathcal{E}$.
We conclude that
\[
\|\varphi(a)\| = \|\varphi_0(a)\| \leq \|a\|_{\mathcal{E}}.\qedhere
\]
\end{proof}

Let $G$ be a locally compact group.
We equip $G$ with a fixed left invariant Haar measure, and we let $L^1(G)$ denote the corresponding group algebra of integrable functions, with product given by convolution.
Note that $L^1(G)$ is a Banach algebra with a contractive, approximate identity.

Let $X$ be a Banach space, and let $\Isom(X)$ denote the group of surjective isometric operators $X\to X$.
We equip $\Isom(X)$ with the strong operator topology.
An \emph{isometric representation} of $G$ on $X$ is a continuous group homomorphism $G\to\Isom(X)$.
Every isometric representation of $G$ on $X$ can be integrated to obtain a nondegenerate contractive representation $L^1(G)\to\Bdd(X)$.
It is well-known that this induces a bijection between isometric representations of $G$ on $X$ and nondegenerate, contractive representations $L^1(G)\to\Bdd(X)$.

Given a class $\mathcal{E}$ of Banach spaces, we denote $F_\mathcal{E}(L^1(G))$ by $F_\mathcal{E}(G)$.
The Banach algebra $F_\mathcal{E}(G)$ captures the representation theory of $G$ on Banach spaces in $\mathcal{E}$.
Given $p\in[1,\infty)$, let $L^p$ denote the class of $L^p$-spaces.
We set $F^p(G)=F_{L^p}(G)$, and we call $F^p(G)$ the \emph{universal group $L^p$-operator algebra} of $G$.
The Banach algebras $F^p(G)$ have been studied in \cite{GarThi15GpAlgLp}.

\begin{cor}
\label{prp:degReprFpG}
Let $p\in(1,\infty)$, and let $G$ be a locally compact group.
Then for each $f\in L^1(G)$, the norm $\|f\|_{L^p}$ agrees with the supremum of $\|\varphi(f)\|$ for all (not necessarily nondegenerate) contractive representations $\varphi\colon L^1(G)\to\Bdd(X)$ on some $L^p$-space $X$.
Consequently, $F^p(G)$ is universal for \emph{all} contractive representations of $L^1(G)$ on $L^p$-spaces.

Analogous statements hold for the universal group $SL^p$-operator algebra (group $QL^p$-operator algebra, group $QSL^p$-operator algebra).
\end{cor}

\autoref{prp:degReprFpG} is a crucial ingredient in \cite{GarThi16arX:ReprConvLq}, where it is used to reduce the question of representability of certain Banach algebras of convolution on $L^q$-spaces to a question about \emph{nondegenerate} representations.
In particular, for every nontrivial, locally compact group $G$, it follows that $F^p(G)$ and $F^q(G)$ are not isometrically isomorphic as Banach algebras for $p,q\in[1,2]$ with $p\neq q$.


\providecommand{\bysame}{\leavevmode\hbox to3em{\hrulefill}\thinspace}
\providecommand{\noopsort}[1]{}
\providecommand{\mr}[1]{\href{http://www.ams.org/mathscinet-getitem?mr=#1}{MR~#1}}
\providecommand{\zbl}[1]{\href{http://www.zentralblatt-math.org/zmath/en/search/?q=an:#1}{Zbl~#1}}
\providecommand{\jfm}[1]{\href{http://www.emis.de/cgi-bin/JFM-item?#1}{JFM~#1}}
\providecommand{\arxiv}[1]{\href{http://www.arxiv.org/abs/#1}{arXiv~#1}}
\providecommand{\doi}[1]{\url{http://dx.doi.org/#1}}
\providecommand{\MR}{\relax\ifhmode\unskip\space\fi MR }
\providecommand{\MRhref}[2]{%
  \href{http://www.ams.org/mathscinet-getitem?mr=#1}{#2}
}
\providecommand{\href}[2]{#2}

\end{document}